\documentclass[11pt]{article}

\usepackage[margin=1.12in]{geometry}
\usepackage[T1]{fontenc}
\usepackage{lmodern}
\usepackage{microtype}
\usepackage{amsmath,amssymb,amsthm,mathtools}
\usepackage{enumitem}
\usepackage{xcolor}
\usepackage[colorlinks=true,linkcolor=blue!55!black,citecolor=blue!55!black,urlcolor=blue!55!black]{hyperref}
\usepackage[capitalise,noabbrev]{cleveref}

\numberwithin{equation}{section}
\newtheorem{theorem}{Theorem}[section]
\newtheorem{proposition}[theorem]{Proposition}
\newtheorem{lemma}[theorem]{Lemma}

\theoremstyle{definition}
\newtheorem{remark}[theorem]{Remark}

\newcommand{\R}{\mathbb R}
\newcommand{\bk}{\backslash}
\newcommand{\C}{\mathbb C}
\newcommand{\bD}{\mathbb D}

\newcommand{\rank}{\operatorname{rank}}
\newcommand{\supp}{\operatorname{supp}}
\newcommand{\vol}{\operatorname{vol}}

\newcommand{\ee}{\mathrm e}
\newcommand{\ii}{\mathrm i}
\newcommand{\dd}{\,\mathrm d}
\newcommand{\norm}[1]{\lVert #1\rVert}
\newcommand{\abs}[1]{\lvert #1\rvert}
\newcommand{\ip}[2]{\langle #1,#2\rangle}
\newcommand{\inj}{\operatorname{inj}}
\newcommand{\ad}{\operatorname{Ad}}

\title{Improved sup-norm bounds for locally symmetric spaces}
\author{Christopher Lutsko}
\date{}

\begin{document}
\raggedbottom
\maketitle

\begin{abstract}
Let $X=G/K$ be a symmetric space of noncompact type, of dimension $n$ and
rank $r$, and let $Y=\Gamma\backslash X$.  Sarnak's local bound
for an $L^2$-normalized spherical joint eigenfunction with regular tempered
parameter of size $T$ is $\norm{\phi}_\infty\ll T^{(n-r)/2}$.  We prove
$o(T^{(n-r)/2})$ locally uniformly on every quotient.  On finite-volume real
hyperbolic manifolds this is uniform in the expanding cusp range
$y\leq T^\beta$, $\beta<1/2$.  If the injectivity radius is bounded below,
we prove the global estimate $T^{(n-r)/2}(\log T)^{-r/2}$. The proof makes use of a novel kernel argument and a uniform bound on the spherical function.

\end{abstract}

\section{Introduction}

Let $G$ be a connected, semisimple Lie group with finite center and maximal
compact subgroup $K<G$. Let $X:=G/K$ be a symmetric space of noncompact type
and let $Y:=\Gamma\bk X$, where $\Gamma<G$ is a torsion-free, discrete
subgroup. Let $W$ denote the restricted Weyl group. Write
\[
 n:=\dim X,\qquad r:=\rank X,\qquad q:=n-r.
 \]
Let $\bD(X)$ denote the algebra of invariant differential operators on $X$.
We say that $\phi\in L^2(Y)$ is a joint eigenfunction if it is a simultaneous
eigenfunction of every operator in $\bD(X)$.  Via the Harish--Chandra
isomorphism, the joint spectrum is parameterized by a $W$-invariant subset
$\sigma\subset\mathfrak a^*_{\C}\cong\C^r$.  To a joint eigenfunction we
associate a spectral parameter $\lambda\in\sigma$, defined up to $W$, and put
$T=1+\norm\lambda$.  A sequence of directions
$\lambda/\norm\lambda$ is called uniformly \emph{regular} if it remains in a
fixed compact subset of an open Weyl chamber.

In a letter to Caroline Morawetz \cite{Sarnak2004}, Peter Sarnak used a
pretrace formula to show that
 \begin{align}
   \|\phi\|_\infty \ll T^{q/2},
 \end{align}
 on compact sets. We show that, along regular rays, this scale is never attained. 

\begin{theorem}[Strict Sarnak bound]\label{thm:main}
Let $\phi_j\in L^2(Y)$ be an $L^2$-normalized spherical joint eigenfunction
with tempered parameter $\lambda_j\in\mathfrak a^*$, and put
$T_j=1+\norm{\lambda_j}$.  Assume that $T_j\to\infty$ and that the directions
$\lambda_j/\norm{\lambda_j}$ are uniformly regular.  Then, on every compact
subset $U\subset Y$,
\begin{equation}\label{eq:local-main}
 \norm{\phi_j}_{L^\infty(U)}=o_{U}(T_j^{q/2}).
\end{equation}
If, in addition, $Y$ is a finite-volume real hyperbolic manifold, fix standard
cusp coordinates and let $Y_{\beta}$ be the compact core together
with the parts of all cusps of height $y\leq T_j^\beta$.  Then, for every
$0<\beta<1/2$,
\begin{equation}\label{eq:truncated-cusp-main}
 \norm{\phi_j}_{L^\infty(Y_{\beta})}
 =o_\beta(T_j^{q/2}).
\end{equation}
\end{theorem}

In fact, we can improve the upper bound by a logarithmic factor as long as the injectivity radius is strictly positive.

\begin{theorem}[Logarithmic improvement]\label{thm:log}
Under the spectral hypotheses of \cref{thm:main}, suppose in addition that
\begin{equation}\label{eq:thickness}
 \inf_{x\in Y}\operatorname{inj}_Y(x)>0.
\end{equation}
Then
\begin{equation}\label{eq:log-main}
 \norm{\phi_j}_{L^\infty(Y)}
 \ll \frac{T_j^{q/2}}{(\log T_j)^{r/2}}.
\end{equation}
The implied constant is uniform for spectral directions in a fixed compact
regular set.  In particular, \eqref{eq:log-main} holds on every compact
quotient and on every convex cocompact real hyperbolic quotient.
\end{theorem}

In proving both our main theorems, a crucial ingredient is Proposition \ref{prop:marshall-tracked} which is a uniform bound on the spherical function. This is an  extension of \cite[Theorem~1.3]{Marshall2016} which is not uniform. In practice, making Marshall's theorem effective involves more than simple book-keeping. Hence, Proposition \ref{prop:marshall-tracked} and its proof may be interesting beyond this application.

\subsection{The sup-norm problem}

For a general compact $n$-manifold and an $L^2$-normalized eigenfunction
$-\Delta u=\lambda^2u$, the local Weyl law gives
\begin{equation}\label{eq:universal-sup}
 \norm{u}_\infty\ll\lambda^{(n-1)/2}.
\end{equation}
The round sphere shows that this can be sharp: zonal harmonics concentrate at
a point.  More generally, Sogge--Zelditch
\cite{SoggeZelditch2002,SoggeZelditch2016} gave precise geometric criteria
for maximal growth.

Negative curvature gives one model for improvement.  B\'erard propagated the
wave kernel to logarithmic time on compact manifolds without conjugate points,
which yields
$\norm{u}_\infty\ll\lambda^{(n-1)/2}(\log\lambda)^{-1/2}$
\cite{Berard1977,HassellTacy2015}; Keeler proved a corresponding two-point
Weyl law \cite{Keeler2024}.  On a
rank-$r$ symmetric space, resolving the full joint spectrum first replaces
the exponent $(n-1)/2$ by Sarnak's $(n-r)/2$. Theorem \ref{thm:log} recovers B\'erard's bound in rank $1$ and saves one factor of $(\log T)^{-1/2}$ for each spectral coordinate.

There is a different, arithmetic route to small sup norms.  Amplification by
Hecke operators gives power savings for special arithmetic eigenbases.  On an
arithmetic hyperbolic surface, Iwaniec--Sarnak proved
\begin{equation}\label{eq:IS-bound}
 \norm{u}_\infty\ll_\varepsilon T^{5/12+\varepsilon},
\end{equation}
against the local exponent $T^{1/2}$ \cite{IwaniecSarnak1995}.  Higher-rank
amplification gives further power savings in particular groups and spectral
or level regimes.  Such estimates use arithmetic correspondences and a Hecke
eigenfunction.  Theorems~\ref{thm:main} and \ref{thm:log} use neither: they apply
to every spherical joint eigenfunction on every quotient in the stated
geometric range.  The two mechanisms are complementary.

The opposite problem of constructing large eigenfunctions is equally
important.  Arithmetic periods and theta lifts produce sequences with
polynomially growing sup norms on compact locally symmetric spaces
\cite{BrumleyMarshall2020,BrumleyMarshall2023}.

\begin{remark}
The multiplier construction is useful beyond point evaluation.  In
particular, in work in progress by the author, the same approach is used to
prove a fixed-window quantum ergodicity theorem in the eigenvalue aspect on a
fixed compact quotient.  This complements fixed-window results in the
Benjamini--Schramm, or level, aspect due to Le Masson--Sahlsten
\cite{LeMassonSahlsten2017} and, in higher rank,
Brumley--Marshall--Matz--Peterson
\cite{BrumleyMarshallMatzPeterson2026}.  The microlocal framework for regular
joint eigenfunctions on compact locally symmetric spaces was developed by
Silberman--Venkatesh \cite{SilbermanVenkatesh2007} and
Hansen--Hilgert--Schr\"oder \cite{HansenHilgertSchroeder2012}.
\end{remark}

\subsection{Proof plan and the Green-kernel viewpoint}

The starting point for this methodology is the observation that in rank $1$,
the Green function plays an important role in estimating spectral mass.  This
can be seen directly in \cite[Chapter~5]{Iwaniec2002} and indirectly in the
point-counting argument of Lax--Phillips \cite{LaxPhillips1982}, later
developed by Kontorovich \cite{Kontorovich2009}, by Kontorovich and the author
\cite{KontorovichLutsko2024}, and in the author's spectral treatment of
horospherical equidistribution \cite{Lutsko2025}.  However, a literal
resolvent is poorly adapted to this problem: at the target eigenvalue it has a
pole and its kernel has global support.  In higher rank, there is also no
canonical one-variable resolvent that resolves all $r$ invariant operators
simultaneously.

The central idea is to replace that resolvent by a positive, smoothed kernel.
Write $\lambda=Tc$, where $c$ is the regular spectral direction.  We choose a
$W$-invariant Paley--Wiener bump $\psi = \psi_{T,c,\epsilon}$ concentrated near the
Weyl orbit of $Tc$ and normalized by $\psi(Tc)=1$.  Taking its
spherical inverse transform $\beta$, we set
\begin{equation}\label{eq:green-surrogate}
 \Phi=\beta^**\beta,
 \qquad \widehat\Phi=\abs{\psi}^2,
 \qquad \widehat\Phi(Tc)=1.
\end{equation}
Thus $\Phi$ has unit response on the eigenfunction under study,
as a normalized
Green kernel would, but it is positive and has finite propagation.  Narrowing
the joint window reduces its value at the identity by the volume of that
window.  The apparent cost is that the spatial support grows.  Regular
spherical oscillation is precisely what makes the nonidentity deck transforms
small enough to pay for that growth.  This trade---spectral narrowing against
spatial propagation---is the proof.

The paper implements the idea in four steps.  Section~\ref{s:3} fixes the
notation and records the harmonic-analytic and geometric preliminaries,
including the qualitative and quantitative spherical decay estimates and the
cusp geometry.  The dependence on the Cartan variable in Marshall's
spherical estimate is proved separately in
Section~\ref{s:marshall-uniformity}; this uniformity is needed when the
propagation radius grows with $T$.  Section~\ref{s:Kernel} then constructs the positive peak by
spherical Paley--Wiener theory, derives the four kernel estimates, and proves
the positive pretrace inequality.  Section~\ref{s:4} first fixes the
bandwidth and then lets it tend to zero, proving the compact-set little-$o$
estimate.  It next sums the rank-one parabolic translates to reach the
expanding cusp range and finally chooses logarithmic propagation to prove
\cref{thm:log}.  The last subsection isolates the Airy transition that
prevents this thick-part argument from becoming an unrestricted global cusp
estimate.

\section{Notation and preliminaries}\label{s:3}

We collect the notation and background used in the construction before
introducing the automorphic kernel.  The identity element of $G$ is denoted by
$e$, and $d_X$ is the invariant distance on $X=G/K$.  For a function $f$ on
$G$, write
\[
 f^*(g)=\overline{f(g^{-1})},\qquad
 (f_1*f_2)(g)=\int_G f_1(gh^{-1})f_2(h)\dd h.
\]
We write $E\Subset Z$ when $E$ is a compact subset of $Z$, and $\supp f$ for
the support of $f$.  The notation $\widehat f$ always denotes the spherical
transform of a $K$-bi-invariant function, while $\norm{\cdot}$ denotes the
metric norm on $\mathfrak a$ or its dual norm on $\mathfrak a^*$.  Implied
constants may depend on the symmetric space and on fixed compact regular sets;
any additional dependence is displayed by a subscript.

\subsection{Structure, measures, and spectral parameters}

Fix a Cartan involution of $G$ and write
$\mathfrak g=\mathfrak k\oplus\mathfrak p$.  Choose a maximal abelian
subspace $\mathfrak a\subset\mathfrak p$, put $A=\exp\mathfrak a$, and fix
an Iwasawa decomposition $G=NAK$.  The Iwasawa projection
$\mathcal A:G\rightarrow\mathfrak a$ is determined by
$g\in N\exp(\mathcal A(g))K$.  The Cartan decomposition is
$G=K\exp(\overline{\mathfrak a^+})K$.  If
$g=k_1\exp H(g)k_2$ with $H(g)\in\overline{\mathfrak a^+}$, then
\begin{equation}\label{eq:cartan-distance}
 d_X(gK,K)=\norm{H(g)}.
\end{equation}
All norms on $\mathfrak a$ and $\mathfrak a^*$ below are the dual norms
induced by the $G$-invariant metric on $X$.

Let $\Sigma^+\subset\mathfrak a^*$ be the positive restricted roots,
$m_\alpha=\dim\mathfrak g_\alpha$, and $W=N_K(A)/Z_K(A)$.  Since $X$ has
no Euclidean factor, the restricted roots span $\mathfrak a^*$.  Hence
\begin{equation}\label{eq:q-root-count}
 \rho=\frac12\sum_{\alpha\in\Sigma^+}m_\alpha\alpha,
 \qquad r=\dim\mathfrak a,\qquad
 q=\sum_{\alpha\in\Sigma^+}m_\alpha=\dim X-r.
\end{equation}
We write $\mathfrak a^*_{\rm reg}$ for the complement of the root
hyperplanes.  A compact set $C\Subset\mathfrak a^*_{\rm reg}$ is always
understood to lie in one open chamber.  Thus there is a number
$\eta_C>0$ such that
\begin{equation}\label{eq:regular-margin}
 \abs{\alpha(c)}\geq\eta_C
 \quad(c\in C,\ \alpha\in\Sigma).
\end{equation}

Haar measure on $G$ is normalized compatibly with the Riemannian measure on
$X=G/K$; the normalization affects only the fixed constants denoted by
$c_X$.  An $L^2$ spherical joint eigenfunction on $Y$ is a smooth right
$K$-invariant function $\phi\in L^2(\Gamma\backslash G)$ which is a joint
eigenvector for $\mathcal D(X)$.  By the Harish--Chandra isomorphism its
eigencharacter is indexed by a $W$-orbit in $\mathfrak a^*_{\mathbb C}$.
For a tempered eigenfunction we choose its representative
$\lambda\in\overline{\mathfrak a^{*+}}$.  Our hypotheses say that
$\lambda=Tc$ with $c\in C$ after replacing $T$ by $\norm\lambda$; the use
of $1+\norm\lambda$ in the statements only accommodates bounded
parameters.

\subsection{Spherical inversion, Paley--Wiener theory, and Plancherel growth}

For $\lambda\in\mathfrak a^*$ the normalized elementary spherical function
is
\begin{equation}\label{eq:elementary-spherical}
 \varphi_\lambda(g)
 =\int_K\exp\!\bigl((\ii\lambda-\rho)
       (\mathcal A(g^{-1}k))\bigr)\dd k,
 \qquad \varphi_\lambda(e)=1.
\end{equation}
Changing $g$ to $g^{-1}$ or changing the sign convention for
$\mathcal A$ gives the equivalent formulas found in the literature.  What
we use is that $\varphi_\lambda$ is $K$-bi-invariant,
$\varphi_{w\lambda}=\varphi_\lambda$, and, for real $\lambda$,
\begin{equation}\label{eq:spherical-unit-bound}
 \abs{\varphi_\lambda(g)}\leq\varphi_0(g)\leq1.
\end{equation}
The first inequality follows directly by taking absolute values in
\eqref{eq:elementary-spherical}.

For $f\in C_c^\infty(K\backslash G/K)$ define
\[
 \widehat f(\lambda)=\int_Gf(g)\varphi_{-\lambda}(g)\dd g.
\]
Spherical inversion and Plancherel's formula are
\begin{align}
 f(g)&=c_X\int_{\mathfrak a^*}\widehat f(\lambda)
       \varphi_\lambda(g)\abs{\mathbf c(\lambda)}^{-2}\dd\lambda,
       \label{eq:spherical-inversion}\\
 \norm{f}_2^2&=c_X\int_{\mathfrak a^*}\abs{\widehat f(\lambda)}^2
       \abs{\mathbf c(\lambda)}^{-2}\dd\lambda.
       \label{eq:spherical-plancherel}
\end{align}
We integrate over all of $\mathfrak a^*$ and absorb $\abs W^{-1}$ into
$c_X$.

We use the following precise form of the spherical Paley--Wiener theorem.
If $F$ is a $W$-invariant entire function on
$\mathfrak a^*_{\mathbb C}$ and, for every $N$, satisfies
\begin{equation}\label{eq:PW-seminorm}
 \abs{F(\lambda+\ii\mu)}
 \leq C_N(1+\norm\lambda+\norm\mu)^{-N}
       \ee^{R\norm\mu},
\end{equation}
then $F=\widehat f$ for a unique
$f\in C_c^\infty(K\backslash G/K)$ supported where
$d_X(gK,K)\leq R$.  Conversely, the transform of every such $f$ satisfies
\eqref{eq:PW-seminorm}.  This is Gangolli's Paley--Wiener theorem
\cite{Gangolli1971}; the equality between exponential type and metric
support uses \eqref{eq:cartan-distance}.

Two consequences of the Gindikin--Karpelevich product formula will be used
repeatedly.  First, globally on the real axis,
\begin{equation}\label{eq:c-global-growth}
 \abs{\mathbf c(\lambda)}^{-2}\leq C_X(1+\norm\lambda)^q.
\end{equation}
More explicitly, if $\Sigma_0^+$ denotes the indivisible positive roots and
$\alpha_0=\alpha/\ip\alpha\alpha$, the product formula and Stirling's
estimate give one-variable factors $p_\alpha$ satisfying
\begin{equation}\label{eq:c-product-size}
 \abs{\mathbf c(\lambda)}^{-2}
 =C_X\prod_{\alpha\in\Sigma_0^+}
 p_\alpha(\ip\lambda{\alpha_0}),
 \qquad
 p_\alpha(s)\ll(1+\abs s)^{m_\alpha+m_{2\alpha}},
\end{equation}
with the matching asymptotic when $\abs s\to\infty$; by convention
$m_{2\alpha}=0$ if $2\alpha$ is not a root.  Summing the exponents in
\eqref{eq:c-product-size} gives
$\sum_{\alpha\in\Sigma_0^+}(m_\alpha+m_{2\alpha})=q$.
Second, for $C\Subset\mathfrak a^*_{\rm reg}$ and
$\norm u\leq T/2$,
\begin{equation}\label{eq:c-regular-growth}
 \abs{\mathbf c(Tc+u)}^{-2}
 \leq C_C T^q(1+\norm u)^q,
 \qquad c\in C.
\end{equation}
For bounded $u$ one has the sharper asymptotic
\begin{equation}\label{eq:c-asymptotic}
 \abs{\mathbf c(Tc+u)}^{-2}
 =T^q\bigl(\beta(c)+O_C(T^{-1}(1+\norm u)^{q+1})\bigr),
\end{equation}
with $\beta$ continuous and strictly positive on $C$.  To see the exponent,
apply Stirling's formula to each gamma quotient in the product formula.
Every positive root contributes its multiplicity to the total degree, so
the degree is the number $q$ in \eqref{eq:q-root-count}.  The same argument
without asymptotic expansion gives \eqref{eq:c-global-growth} and
\eqref{eq:c-regular-growth}.

\subsection{Uniform decay of regular spherical functions}

The analytic input is Marshall's root-product estimate.  We state exactly
the two consequences needed below, including the dependence on a growing
group-variable compactum.

\begin{proposition}[Root-product estimate with tracked constants]
\label{prop:marshall-tracked}
Let $C_1\Subset\mathfrak a^*_{\rm reg}$.  There are constants $A$ and $D$
depending only on $X$ and $C_1$ such that, for $R\geq1$, $t\geq1$,
$c\in C_1$, and $H\in\mathfrak a$ with $\norm H\leq R$,
\begin{equation}\label{eq:marshall-tracked}
 \abs{\varphi_{tc}(\exp H)}
 \leq D\ee^{AR}
 \prod_{\alpha\in\Sigma^+}
       (1+t\abs{\alpha(H)})^{-m_\alpha/2}.
\end{equation}
\end{proposition}

\begin{proof}
The dependence on $R$ is proved in \cref{s:marshall-uniformity}.
\end{proof}

\begin{lemma}[Decay away from the identity]\label{lem:spherical-decay}
Let $C\Subset\mathfrak a^*_{\mathrm{reg}}$, let
$E\Subset G\setminus K$, and let $U\Subset\mathfrak a^*$.  Then
\begin{equation}\label{eq:spherical-decay}
 \sup_{c\in C,\,u\in U,\,g\in E}
 \abs{\varphi_{Tc+u}(g)}\longrightarrow0
 \qquad(T\to\infty).
\end{equation}
\end{lemma}

\begin{proof}
By Cartan decomposition, $E$ maps to a compact set $B_E\subset\mathfrak a$
which does not contain $0$.  Hence there is $\iota_E>0$ with
$\norm H\geq\iota_E$ for $H\in B_E$.  Since the restricted roots span
$\mathfrak a^*$, norm equivalence gives a constant $c_X>0$ such that
\begin{equation}\label{eq:root-detects-H}
 \max_{\alpha\in\Sigma^+}\abs{\alpha(H)}
 \geq c_X\norm H.
\end{equation}
For $u\in U$ write $Tc+u=T(c+u/T)$.  Once $T$ is large, the directions
$c+u/T$ lie in a fixed compact regular set $C_1$.  Apply
\cref{prop:marshall-tracked} with a fixed ball containing $B_E$.  Choose a
root $\alpha_H$ satisfying \eqref{eq:root-detects-H} and discard every
factor in the product except its factor.  If
$m_0=\min_{\alpha\in\Sigma^+}m_\alpha$, then
\[
 \abs{\varphi_{Tc+u}(\exp H)}
 \leq C_E(1+Tc_X\iota_E/2)^{-m_0/2}
 \leq C_E'T^{-m_0/2}.
\]
This is uniform in $c,u,H$ and tends to zero, proving
\eqref{eq:spherical-decay}.
\end{proof}

\begin{lemma}[Quantitative regular decay]\label{lem:quantitative-decay}
Fix $C\Subset\mathfrak a^*_{\mathrm{reg}}$ and $\iota>0$.  There are
$\delta>0$ and constants $A,M$ such that, whenever
\[
 \iota\leq\norm H\leq A_0\log T,
 \qquad c\in C,
 \qquad \norm u\leq T^{1/2},
\]
one has
\begin{equation}\label{eq:quantitative-decay}
 \abs{\varphi_{Tc+u}(\exp H)}
 \leq C_{C,\iota,A_0}(1+\norm u)^M\ee^{A\norm H}T^{-\delta}.
\end{equation}
\end{lemma}

\begin{proof}
Write $Tc+u=T(c+u/T)$.  For $\norm u\leq T^{1/2}$, the vector
$c+u/T$ belongs, for all sufficiently large $T$, to a fixed compact regular
set $C_1$ containing $C$ in its interior.  Apply
\cref{prop:marshall-tracked} with
$R_H=\max(1,\norm H)$; the hypothesis on $H$ ensures
$R_H\leq1+A_0\log T$.  By
\eqref{eq:root-detects-H}, some positive root satisfies
$\abs{\alpha(H)}\geq c_X\iota$.  Keeping only that factor gives
\[
 \abs{\varphi_{Tc+u}(\exp H)}
 \leq D'\ee^{A\norm H}
 (1+Tc_X\iota)^{-m_0/2},
 \qquad m_0=\min_{\alpha>0}m_\alpha.
\]
Thus \eqref{eq:quantitative-decay} holds with
$\delta=m_0/2$ and $M=0$.  Enlarging the constant covers the bounded range
of $T$ for which $c+u/T$ has not yet entered $C_1$.
\end{proof}

\subsection{Orbit geometry}

We shall use the following elementary geometric observation.

\begin{lemma}[Uniform local orbit bounds]\label{lem:orbit}
Let $K_0\Subset Y$ and $R<\infty$.  There are $\iota>0$ and $N<\infty$ such
that, for every $x\in K_0$ and every lift $g_xK$,
\begin{align}
 d_X(g_xK,\gamma g_xK)&\geq\iota\quad(\gamma\ne e),
 \label{eq:orbit-separation}\\
 \#\{\gamma:d_X(g_xK,\gamma g_xK)\leq R\}&\leq N.
 \label{eq:orbit-count}
\end{align}
If $\inf_Y\operatorname{inj}_Y>0$, the constants are uniform for
$K_0=Y$; moreover one may take
\begin{equation}\label{eq:orbit-exponential}
 N\leq C_\iota\ee^{h_XR}
\end{equation}
for a constant $h_X$ depending only on $X$.
\end{lemma}

\begin{proof}
For a torsion-free quotient,
\[
 2\inj_Y(x)=\inf_{\gamma\ne e}d_X(g_xK,\gamma g_xK).
\]
The injectivity-radius function is continuous.  It therefore has a positive
minimum on $K_0$, proving \eqref{eq:orbit-separation} with
$\iota=2\min_{K_0}\inj_Y$.

Put $p=g_xK$.  Distinct orbit points $\gamma p$ and $\gamma'p$ are at
distance
\[
 d_X(\gamma p,\gamma'p)
 =d_X(p,\gamma^{-1}\gamma'p)\geq\iota.
\]
Thus the balls $B_X(\gamma p,\iota/3)$ are disjoint.  If
$d_X(p,\gamma p)\leq R$, each such ball lies in
$B_X(p,R+\iota/3)$.  Homogeneity of $X$ now gives the explicit packing
bound
\begin{equation}\label{eq:packing-ratio}
 \#\{\gamma:d_X(p,\gamma p)\leq R\}
 \leq\frac{\vol B_X(p,R+\iota/3)}
          {\vol B_X(p,\iota/3)}.
\end{equation}
This proves \eqref{eq:orbit-count}.  The polar-coordinate formula on a
noncompact symmetric space shows
$\vol B_X(p,s)\leq C_X\ee^{h_Xs}$ for $s\geq1$ (one may take any
$h_X>2\norm\rho$).  Inserting this in \eqref{eq:packing-ratio} proves
\eqref{eq:orbit-exponential}.  If $\inf_Y\inj_Y>0$, the same $\iota$ works
for every $x\in Y$, so every step is global.
\end{proof}

\subsection{Real-hyperbolic cusp coordinates}

For the global part of \cref{thm:main}, let $X=\mathbb H^{q+1}$ in the
upper-half-space model.  A finite-volume quotient has a compact core and
finitely many cusps.  After passing to a finite cover of a cusp cross-section
(which changes every estimate only by the degree of that cover), one cusp is
\begin{equation}\label{eq:cusp-model-prelim}
 \mathcal C=(\Lambda\backslash\mathbb R^q)\times[y_0,\infty),
 \qquad
 ds^2=y^{-2}(\dd y^2+\dd z^2),\qquad
 \dd\vol=y^{-q-1}\dd z\dd y,
\end{equation}
where $\Lambda$ is a full Euclidean lattice.  The parabolic displacement is
given exactly by
\begin{equation}\label{eq:parabolic-exact-distance}
 \cosh d((z,y),(z+v,y))=1+\frac{\abs v^2}{2y^2},
 \qquad
 d((z,y),(z+v,y))=2\operatorname{arsinh}\frac{\abs v}{2y}.
\end{equation}
In particular, for each fixed $R$,
\begin{equation}\label{eq:cusp-distance-prelim}
 d((z,y),(z+v,y))\asymp_R\frac{\abs v}{y}
 \quad\text{when }d((z,y),(z+v,y))\leq R.
\end{equation}

Choose the initial horoball precisely invariant under its parabolic
stabilizer $\Gamma_\infty$.  The depth function
$b(z,y)=\log(y/y_0)$ is a Busemann function and therefore is $1$-Lipschitz.
If $\gamma\notin\Gamma_\infty$, the interiors of the chosen horoball and
its $\gamma$-translate are disjoint.  A path from $(z,y)$ to
$\gamma(z,y)$ must descend through depth zero in the first horoball and
enter the second horoball from depth zero.  The Lipschitz property on the two
segments gives
\begin{equation}\label{eq:nonparabolic-distance-prelim}
 d((z,y),\gamma(z,y))\geq2\log(y/y_0).
\end{equation}
Replacing the chosen cusp neighborhood by a boundedly different one gives
the equivalent form $2\log y-O(1)$ used later.

Take the nonnegative Laplacian.  A tempered eigenfunction has eigenvalue
$q^2/4+T^2$.  Expanding in the characters of the torus and solving the
resulting ordinary differential equation gives
\begin{equation}\label{eq:Fourier-expansion-prelim}
 \phi(z,y)=a_0^+y^{q/2+\ii T}+a_0^-y^{q/2-\ii T}
 +y^{q/2}\sum_{0\ne m\in\Lambda^*}
 a_mK_{\ii T}(2\pi\abs m y)\ee^{2\pi\ii\ip mz}.
\end{equation}
Indeed, substituting $y^{q/2}F(2\pi\abs m y)$ into the separated radial
equation gives
\[
 u^2F''(u)+uF'(u)-(u^2-T^2)F(u)=0,
\]
whose solution decaying as $u\to\infty$ is $K_{\ii T}(u)$.  The other
solution $I_{\ii T}(u)$ grows exponentially and is excluded by $L^2$.
For the constant term, put $s=\log y$.  If $T>0$, then
\begin{align*}
 &\int_{\log y_0}^{S}
 \abs{a_0^+\ee^{\ii Ts}+a_0^-\ee^{-\ii Ts}}^2\dd s\\
 &\qquad=(\abs{a_0^+}^2+\abs{a_0^-}^2)(S-\log y_0)+O_T(1).
\end{align*}
This is the radial $L^2$ integral after the factors $y^{q/2}$ and
$y^{-q-1}\dd y$ are combined.  It stays bounded as $S\to\infty$ only when
\begin{equation}\label{eq:constant-term-zero}
 a_0^+=a_0^-=0
\end{equation}
for every $L^2$ eigenfunction with real $T$.

Parseval on $\Lambda\backslash\mathbb R^q$, followed by the substitution
$u=2\pi\abs m y$, gives
\begin{equation}\label{eq:cusp-parseval-prelim}
 \int_{\mathcal C}\abs\phi^2\dd\vol
 =\vol(\Lambda\backslash\mathbb R^q)
 \sum_{m\ne0}\abs{a_m}^2I_T(m),
 \quad
 I_T(m)=\int_{2\pi\abs m y_0}^{\infty}
 \abs{K_{\ii T}(u)}^2\frac{\dd u}{u}.
\end{equation}
This identity is the source of the simultaneous control of all Fourier
modes in the far cusp.

\section{Uniformity in the spherical estimate}
\label{s:marshall-uniformity}

For a fixed compact set in the Cartan variable,
\cref{prop:marshall-tracked} is Marshall's
\cite[Theorem~1.3]{Marshall2016}.  The point here is the explicit dependence
on the ball $B_R$.  This is needed below when the support of the kernel grows
with $T$.  We give the argument in full.  The local normal forms can be
quantified directly; the additional issue is the region away from them.  It
is handled by applying semialgebraic separation to the total critical
incidence in the phase and Cartan variables.

Throughout this section, fix
$C_1\Subset\mathfrak a^*_{\rm reg}$ and put
\[
 B_R=\{H\in\mathfrak a:\norm H\leq R\},
 \qquad M=Z_K(A),\qquad M'=N_K(A).
\]
Thus $W=M'/M$.

\subsection{Quantitative lemmas}

For a function on a fixed coordinate box $U$, write
\[
 \norm f_{C^N(U)}=
 \max_{\abs\nu\leq N}\sup_{x\in U}\abs{\partial^\nu f(x)}.
\]

\begin{lemma}[Quantitative coordinates]\label{lem:quantitative-calculus}
Fix $N$ and let $M\geq2$.
\begin{enumerate}[label=\textup{(\roman*)}]
\item If $f(0,y)=0$ and $\norm f_{C^{N+1}}\leq M$, then $f/x$ extends
smoothly across $x=0$ and has $C^N$ norm at most $M$.  Division by a product
of $q$ coordinate functions costs at most $q$ derivatives.
\item If a $C^{N+1}$ map has norm at most $M$ and the inverse of the minor
used in the inverse or implicit function theorem has norm at most $M$, then
the resulting map and its inverse are defined on a box of radius
$c_NM^{-B_N}$ and have $C^N$ norm at most $C_NM^{B_N}$.
\item If $V$ is a $C^{N+1}$ vector field of norm at most $M$, its flow is
defined for $|s|\leq cM^{-1}$ on the half-sized box and has $C^N$ norm at
most $C_NM^{B_N}$.  If a hypersurface is transverse to $V$ with transverse
derivative at least $M^{-1}$, the associated flow coordinates and their
inverse satisfy the same bounds on a box of radius $c_NM^{-B_N}$.
\end{enumerate}
The constants depend only on the fixed dimensions and $N$.
\end{lemma}

\begin{proof}
For (i), differentiate
\[
 \frac{f(x,y)}x=\int_0^1(\partial_xf)(sx,y)\dd s,
\]
and iterate.  For (ii), translate the base point and first take one implicit
variable with $|\partial_yF(0,0)|\geq M^{-1}$.  On a box of radius
$(4M^2)^{-1}$ this derivative remains at least $(2M)^{-1}$.  The contraction
map
\[
 y\longmapsto y-(\partial_yF(0,0))^{-1}F(x,y)
\]
then works for $|y|\leq(16M^2)^{-1}$ and
$|x|\leq(64M^4)^{-1}$.  Differentiating $F(x,h(x))=0$ and using Cramer's
rule gives the stated polynomial bounds.  Finally, the integral equation
for the flow and Gronwall's inequality give (iii); the coordinate assertion
follows from (ii).
\end{proof}

\begin{lemma}[Oscillatory integral]\label{lem:osc}
Let $a\in C_c^{d+L}((-1,1)^{d+e})$, and let
$q_1,\ldots,q_d,Q\in\R$, where $L$ is a nonnegative integer.  Then
\begin{align}
 &\left|\int a(x,y)
 \exp\left(\ii t\left(\sum_{j=1}^dq_jx_j^2+Qy_1\right)\right)
 \dd x\dd y\right|\notag\\
 &\qquad\leq C_{d,e,L}\norm a_{C^{d+L}}
 (1+t\abs Q)^{-L}
 \prod_{j=1}^d(1+t\abs{q_j})^{-1/2}.
 \label{eq:osc}
\end{align}
If the linear term is absent, its factor is omitted.
\end{lemma}

\begin{proof}
For one quadratic variable, estimate the interval
$\abs x\leq(1+t\abs q)^{-1/2}$ by its length and use
\[
 \ee^{\ii tqx^2}=(2\ii tqx)^{-1}\partial_x\ee^{\ii tqx^2}
\]
on the complement.  The result is
$C\norm a_{C^1}(1+t\abs q)^{-1/2}$.  Iterate in the quadratic variables
and integrate $L$ times in $y_1$ when $t|Q|\geq1$.
\end{proof}

We also use one standard fact from real algebraic geometry.  Recall that a
Nash function is a real analytic semialgebraic function.

\begin{lemma}[Global Nash separation]\label{lem:nash-separation}
Let $X\subset\R^n$ be a closed semialgebraic set, let
$F:X\rightarrow\R^m$ be continuous and semialgebraic, and put
$Z=F^{-1}(0)$.  Assume that $Z$ is nonempty.  There are constants $C,N>0$
such that
\begin{equation}\label{eq:nash-separation}
 \norm{F(x)}\geq C^{-1}(1+\norm x)^{-N}
 \min\{1,\operatorname{dist}(x,Z)\}^{N}
 \qquad (x\in X).
\end{equation}
\end{lemma}

\begin{proof}
For $T,u\geq1$, consider the compact semialgebraic set
\[
 E_{T,u}=\{x\in X:\norm x\leq T,
       \operatorname{dist}(x,Z)\geq u^{-1}\}.
\]
Let $m(T,u)$ be the minimum of $\norm F$ on $E_{T,u}$, with value $1$ if
the set is empty.  By Tarski--Seidenberg, $m$ is semialgebraic and positive.
On each cell of continuity, polynomial growth at infinity gives
\[
 m(T,u)^{-1}\leq C(1+T+u)^N.
\]
Taking $T=1+\norm x$ and
$u=\max\{1,\operatorname{dist}(x,Z)^{-1}\}$ proves
\eqref{eq:nash-separation}.  See
\cite[Proposition~2.6.2 and Corollary~2.6.7]{BCR}.
\end{proof}

\subsection{The phase and amplitude}

Use the $NAK$ Iwasawa decomposition
\begin{equation}\label{eq:Iwasawa}
 g=n(g)\exp\mathcal A(g)k(g).
\end{equation}
The Harish--Chandra formula is
\begin{equation}\label{eq:HC}
 \varphi_{tc}(\exp H)
 =\int_K b(k,H)\ee^{\ii t\Theta(k,H,c)}\dd k,
 \quad
 \Theta(k,H,c)=c(\mathcal A(k\exp H)),
 \quad
 b(k,H)=\ee^{\rho(\mathcal A(k\exp H))}.
\end{equation}

\begin{lemma}\label{lem:phase}
For every fixed $N$,
\begin{equation}\label{eq:phase-bound}
 \norm\Theta_{C^N(K\times B_R\times C_1)}
 +\norm b_{C^N(K\times B_R)}
 \leq C_N\ee^{C_NR}.
\end{equation}
\end{lemma}

\begin{proof}
After passing to a finite cover of $G$, choose finite-dimensional
representations $\pi_j$ with restricted highest weights $\omega_j$ spanning
$\mathfrak a^*$.  Give their spaces $K$-invariant norms and choose vectors
$v_j$ satisfying
\[
 \pi_j(n)v_j=v_j,
 \qquad
 \pi_j(\exp U)v_j=\ee^{\omega_j(U)}v_j.
\]
Inserting \eqref{eq:Iwasawa} and using the $K$-invariant norm gives
\begin{equation}\label{eq:weight-formula}
 \ee^{\omega_j(\mathcal A(g))}
 =\frac{\norm{v_j}}{\norm{\pi_j(g^{-1})v_j}}.
\end{equation}
For $g=k\exp H$ with $\norm H\leq R$,
\begin{equation}\label{eq:rep-bound}
 \norm{\pi_j(g)}+\norm{\pi_j(g^{-1})}\leq C\ee^{CR}.
\end{equation}
Put
$F_j(k,H)=\norm{\pi_j((k\exp H)^{-1})v_j}^2$.  Equation
\eqref{eq:rep-bound} and
$\norm{v_j}\leq\norm{\pi_j(g)}\norm{\pi_j(g^{-1})v_j}$ give
\begin{equation}\label{eq:F-bound}
 C^{-1}\ee^{-CR}\leq F_j\leq C\ee^{CR},
 \qquad
 \norm{F_j}_{C^N}\leq C_N\ee^{C_NR}.
\end{equation}
Taking logarithms in \eqref{eq:weight-formula},
\begin{equation}\label{eq:logF}
 \omega_j(\mathcal A(k\exp H))
 =\log\norm{v_j}-\frac12\log F_j(k,H).
\end{equation}
Every fixed derivative of $\log F_j$ is a sum of terms
\[
 F_j^{-q}\prod_{i=1}^q\partial^{\nu_i}F_j.
\]
Equation \eqref{eq:F-bound} bounds each such term by $C_N\ee^{C_NR}$.
Since the $\omega_j$ span $\mathfrak a^*$, the same bound holds for the
derivatives of $\mathcal A$.  Applying the chain rule to the phase and amplitude
in \eqref{eq:HC} proves \eqref{eq:phase-bound}.
\end{proof}

We next record the algebraic property that will replace the qualitative
compactness argument on the nonstationary region.  Choose a basis
$Y_1,\ldots,Y_d$ of $\mathfrak k$, viewed as left-invariant vector fields,
and set
\begin{equation}\label{eq:gradient-map}
 \mathcal F(k,a,c)=
 \bigl(Y_1\Theta(k,\log a,c),\ldots,
       Y_d\Theta(k,\log a,c)\bigr),
 \qquad a\in A.
\end{equation}
The notation in \eqref{eq:gradient-map} is only used to identify
$a=\exp H$; the logarithm does not enter the assertion below.

\begin{lemma}\label{lem:nash-gradient}
There is a compact semialgebraic set $C_\eta\Subset\mathfrak a^*_{\rm reg}$
containing $C_1$ such that $\mathcal F$ is a Nash map on
$K\times A\times C_\eta$.  Its zero set
\begin{equation}\label{eq:critical-incidence}
 \mathcal Z=\{(k,a,c):\mathcal F(k,a,c)=0\}
\end{equation}
is the total critical incidence.  In a fixed algebraic embedding of
$K\times A\times C_\eta$ into Euclidean space,
\begin{equation}\label{eq:height-bound}
 \norm{(k,\exp H,c)}\leq C\ee^{CR}
 \qquad (\norm H\leq R).
\end{equation}
Moreover, if
\begin{equation}\label{eq:H-incidence}
 \mathcal Z_H=\{(k,H,c):k\in M'K_H\},
\end{equation}
then $\mathcal Z_H$ is closed and semialgebraic.  The map
$H\mapsto\exp H$ between $B_R$ and the algebraic $A$-coordinates used above
is bi-Lipschitz, with both Lipschitz constants at most $C\ee^{CR}$.
\end{lemma}

\begin{proof}
Let $M_0$ bound $\norm c$ on $C_1$, and let $\eta>0$ be smaller than the
minimum in \eqref{eq:regular-lower} below.  We may take
\[
 C_\eta=\{c:\norm c\leq M_0+1,
       \abs{\ip c\alpha}\geq\eta\text{ for every }\alpha\in\Sigma\}.
\]
This is compact and semialgebraic.

After passing to the same finite cover as above, regard $G$ as a real
algebraic matrix group.  The multiplication map $N\times A\times K\to G$
is a Nash diffeomorphism, and hence its inverse, including the $A$-component
$\mathbf a(g)$, is Nash.  Choose algebraic characters
$\chi_1,\ldots,\chi_r$ of $A$ whose differentials span $\mathfrak a^*$.
If $a=\mathbf a(g)$, then
\[
 c(\log a)=\sum_{j=1}^r c_j\log\chi_j(a)
\]
for suitable linear coordinates $c_j$ on $\mathfrak a^*$.  Differentiating
in a $K$-direction replaces $\log\chi_j(a)$ by
$d\chi_j(a)/\chi_j(a)$, which is Nash.  Thus every component in
\eqref{eq:gradient-map} is Nash.  Its zero set is exactly the set where the
$K$-gradient of the phase vanishes, proving \eqref{eq:critical-incidence}.
Marshall's Proposition~4.3 identifies its fibre over $(\exp H,c)$ with
$M'K_H$.

The condition $u\in K_H$ is the algebraic equation $\ad(u)H=H$.
Since $M'$ and $K$ are compact algebraic groups, \eqref{eq:H-incidence} is
closed and semialgebraic.  In character coordinates the derivatives of
$H\mapsto\exp H$ and of its inverse are bounded by $C\ee^{CR}$ on $B_R$.

Embed $A$ by the coordinates
$\chi_j(a)$ and $\chi_j(a)^{-1}$.  The image is closed and semialgebraic.
Every character occurring here is the exponential of a fixed linear form
in $H$, so \eqref{eq:height-bound} follows.  The $K$ and $c$ coordinates
range over fixed compact sets.
\end{proof}

\subsection{Normal-form inputs}

We use the following parts of \cite[Section~4]{Marshall2016}.

Marshall resolves the root hyperplanes in $\mathfrak a$ by finitely many
charts associated with flags $F$.  In a resolved chart the coordinates are
$z_0,\ldots,z_{r-1}$ with
\begin{equation}\label{eq:z-domain}
 0\leq z_0\leq C(1+R),
 \qquad 0\leq z_j\leq1\quad(j\geq1),
\end{equation}
and every root has the form
\begin{equation}\label{eq:root-resolution}
 \alpha(H)=u_{\alpha,F}(z)z_0z_1\cdots z_{j(\alpha,F)-1},
\end{equation}
where $u_{\alpha,F}$ and its inverse have bounded derivatives of every fixed
order.  This is \cite[Section~4.2.2 and Lemma~4.6]{Marshall2016}.

The quadratic induction divides a phase derivative by the corresponding
root coordinate.  In the notation of \cite[Proposition~4.9]{Marshall2016},
\begin{equation}\label{eq:Marshall-psi}
 \psi=(w^{-1}\beta_X)^{-1}V_\beta^+\Theta_1.
\end{equation}
The numerator vanishes on the divisor $w^{-1}\beta_X=0$, and the derivative
inverted in the implicit function theorem is
\begin{equation}\label{eq:Marshall-pivot}
 V_\beta^+\psi(l,s)
 =-\ip c\beta\ee^x\frac{\sinh x}{x},
 \qquad x=w^{-1}\beta(H),
\end{equation}
with its continuous value at $x=0$.  See
\cite[Lemmas~4.10--4.13]{Marshall2016}; the same normalized coefficient is
used for the Morse coordinate.

For $l\in K$, after all quadratic variables have been removed, the residual phase is
divided by the monomial $Q=z_0\cdots z_q$ defining the current Levi stratum.
The nonvanishing derivative needed to make the final linear coordinate is
given by \cite[Lemmas~4.16--4.19]{Marshall2016}.  It has the form
\begin{equation}\label{eq:Marshall-linear}
 \alpha(H^L)
 \ip{\ad(l^{-1})H_c}{V_{-\alpha}-V_\alpha},
\end{equation}
and at least one of these expressions is nonzero.  Here $H_c\in\mathfrak a$
is dual to $c$, $L$ is the Levi subgroup associated with $l$, and $H^L$ is
the normalized direction transverse to the split centre of $L$.

The resulting normal forms are \cite[Theorem~4.7 and
Proposition~4.22]{Marshall2016}.  At a Weyl point,
\begin{equation}\label{eq:Weyl-normal}
 \Theta=\Theta_0(H,c)
 -\sum_{\alpha\in\Sigma^+}\sum_{j=1}^{m_\alpha}
  \ip c\alpha\alpha(wH)x_{\alpha,j}^2.
\end{equation}
and away from the Weyl points,
\begin{equation}\label{eq:nonWeyl-normal}
 \Theta=\Theta_0(H,c)
 -\sum_{(\alpha,j)\in\mathcal R}
  \ip c\alpha\alpha(H)x_{\alpha,j}^2+Q(H)L(y),
\end{equation}
where $L$ is a nonconstant affine-linear function and $\mathcal R$ is the
set of quadratic root directions on that stratum.  Marshall's
\cite[Proposition~4.20]{Marshall2016} also gives, on each such chart,
\begin{equation}\label{eq:root-comparison}
 \abs{\alpha(H)}\leq C\abs{Q(H)}
 \quad\text{for every root direction not in }\mathcal R.
\end{equation}
The comparison follows from the smooth quotient $\alpha/Q$ on the resolved
face and, in the remaining directions, from
$|\alpha(H)|\ll\norm{H^L}\asymp|Q|$ in the proof of that proposition.
Both comparisons depend only on the angular variables, so the constant is
independent of $R$.

\subsection{Quantitative coordinate bounds}

We now deduce the dependence on $R$ from the preceding inputs.

First pull the phase and amplitude back to the resolved coordinates
\eqref{eq:z-domain}.  The resolution map is polynomial in $z_0$ and smooth
in the other variables.  Its fixed derivatives are therefore polynomial in
$1+R$, and hence bounded by $C_N\ee^{C_NR}$.  Together with
\eqref{eq:phase-bound}, this gives
\begin{equation}\label{eq:resolved-phase}
 \norm\Theta_{C^N}+\norm b_{C^N}\leq C_N\ee^{C_NR}
\end{equation}
on every resolved chart.

Consider one step of the induction \eqref{eq:Marshall-psi}.  By
\eqref{eq:root-resolution}, the divisor $w^{-1}\beta_X$ is a monomial in at
most $r$ resolved coordinates times a smooth unit.  The numerator vanishes
on every component of that divisor by Marshall's Lemma~4.10.  Applying
part (i) of \cref{lem:quantitative-calculus} at most $r$ times to
\eqref{eq:resolved-phase} gives
\begin{equation}\label{eq:psi-bound}
 \norm\psi_{C^N}\leq C_N\ee^{C_NR}.
\end{equation}
This proves the required upper bound for the divided function without
introducing a factor $\abs{\beta(H)}^{-1}$.

For the inverse derivative, put
$p(x)=\ee^x\sinh(x)/x$, with $p(0)=1$.  Since
$p(x)=(\ee^{2x}-1)/(2x)$, one has $p(x)\geq1$ for $x\geq0$ and
$p(x)^{-1}\ll1+|x|$ for $x<0$.  As $|x|\ll R$ and
\begin{equation}\label{eq:regular-lower}
 \eta:=\min_{c\in C_1}\min_{\alpha\in\Sigma}
 \abs{\ip c\alpha}>0,
\end{equation}
the pivot \eqref{eq:Marshall-pivot} satisfies
\begin{equation}\label{eq:pivot-bounds}
 \abs{(V_\beta^+\psi)^{-1}}\leq C(1+R),
 \qquad
 \norm{V_\beta^+\psi}_{C^N}\leq C_N\ee^{C_NR}.
\end{equation}
The derivative bound also follows from \eqref{eq:psi-bound}.

Parts (ii) and (iii) of \cref{lem:quantitative-calculus}, applied with
\eqref{eq:psi-bound} and \eqref{eq:pivot-bounds}, give exponential bounds
for the implicit, flow, and Morse coordinates in
\cite[Lemmas~4.12--4.13]{Marshall2016}.  Taking a square root in the Morse
step has the same cost because the normalized coefficient and its inverse
have exponential seminorms.  There are at most
$d_0=\sum_{\alpha>0}m_\alpha$ steps, independently of $R$.  Hence all
quadratic coordinate maps and their inverses have $C^N$ norm at most
\begin{equation}\label{eq:quadratic-map-bound}
 C_N\ee^{C_NR}
\end{equation}
on boxes of radius at least $\ee^{-C_NR}$.

The regular trivialization in \cite[Theorem~4.7]{Marshall2016} has the same
bounds.  Indeed, use the surviving $K$-coordinates and the unchanged
parameter coordinates on the submanifold of Corollary~4.15.  Its Jacobian
is block triangular, with identity on the parameter block and the preceding
coordinate changes on the other block.

For the final linear coordinate, divide the residual phase by
$Q=z_0\cdots z_q$.  It vanishes on every component of $Q=0$, so part (i)
of \cref{lem:quantitative-calculus} and
\eqref{eq:quadratic-map-bound} bound the quotient by $C_N\ee^{C_NR}$.

Choose Weyl neighbourhoods $\mathcal U_0\Subset\mathcal U_1$ and use
\cite[Proposition~4.22]{Marshall2016} on $\mathcal U_1$.  On the remaining
charts, $l\in K\setminus\mathcal U_0$.  After $Q$ has been removed, the
quantity \eqref{eq:Marshall-linear} depends only on $l$, the resolved
angular variables, and $c$.  These variables range over a compact set for
each of the finitely many flag charts and Levi types.  Lemma~4.19 of
\cite{Marshall2016} says that at least one of the quantities is nonzero.
A finite subcover therefore gives
\begin{equation}\label{eq:linear-lower}
 \max_\alpha
 \left|
 \alpha(H^L)
 \ip{\ad(l^{-1})H_c}{V_{-\alpha}-V_\alpha}
 \right|\geq\kappa
\end{equation}
for some $\kappa>0$ independent of $R$.  Part (ii) of
\cref{lem:quantitative-calculus} now gives the same exponential bounds for
the final linear coordinate and its inverse.


\subsection{The local integrals}

We now apply the above analysis to the normal forms
\eqref{eq:Weyl-normal} and \eqref{eq:nonWeyl-normal}.

In a Weyl chart, apply \cref{lem:osc} to
\eqref{eq:Weyl-normal}.  By \eqref{eq:regular-lower}, the coefficient
$\ip c\alpha$ can be absorbed in the implied constant.  There are
$m_\alpha$ quadratic variables for each $\alpha$.  Since $w$ permutes the
restricted roots and preserves their multiplicities, relabelling the roots
replaces the factors $\abs{\alpha(wH)}$ by $\abs{\alpha(H)}$.  Therefore the
chart contributes at most
\begin{equation}\label{eq:local-product}
 C\ee^{CR}
 \prod_{\alpha\in\Sigma^+}
       (1+t\abs{\alpha(H)})^{-m_\alpha/2}.
\end{equation}

In a non-Weyl chart, make a fixed linear change of the residual variables
so that the nonconstant affine-linear function $L(y)$ in
\eqref{eq:nonWeyl-normal} is the first coordinate.  Its Jacobian and inverse
are uniformly bounded by \eqref{eq:linear-lower}.  Lemma~\ref{lem:osc}
then gives, for every fixed integer $L_0$,
\begin{equation}\label{eq:local-preproduct}
 C_{L_0}\ee^{C_{L_0}R}
 (1+t\abs Q)^{-L_0}
 \prod_{(\alpha,j)\in\mathcal R}
       (1+t\abs{\alpha(H)})^{-1/2}.
\end{equation}
Choose
\begin{equation}\label{eq:L0}
 L_0\geq\frac12
 \sum_{(\alpha,j)\notin\mathcal R}1.
\end{equation}
By \eqref{eq:root-comparison},
\[
 1+t\abs{\alpha(H)}\leq C(1+t\abs Q)
 \qquad ((\alpha,j)\notin\mathcal R).
\]
Consequently,
\begin{equation}\label{eq:Q-pays}
 (1+t\abs Q)^{-L_0}
 \leq C
 \prod_{(\alpha,j)\notin\mathcal R}
       (1+t\abs{\alpha(H)})^{-1/2}.
\end{equation}
Equations \eqref{eq:local-preproduct} and \eqref{eq:Q-pays} give
\eqref{eq:local-product}.

It remains to estimate the complement of the normal-form boxes.  The boxes
are constructed along the total critical incidence $\mathcal Z$ in
\eqref{eq:critical-incidence}.  Thus points that become almost critical as
$H$ approaches a root wall are included in boxes based on that wall.

Take concentric inner and outer boxes.  If $r_R\geq\ee^{-C_1R}$ is an outer
radius and $L_R\leq\ee^{C_2R}$ bounds the $C^1$ norm of the coordinate map,
then the inner boxes contain the $r_R/(4L_R)$-neighbourhood of the lifted
incidence.  The blowdown map and $\mathcal Z_H$ in
\eqref{eq:H-incidence} are semialgebraic.  Applying
\cref{lem:nash-separation} to the pullback of the distance from
$\mathcal Z_H$, and using that the resolved radial coordinate is $O(1+R)$,
shows that the images of the inner boxes contain the
\begin{equation}\label{eq:critical-tube}
 \ee^{-C R}\text{-neighbourhood of }\mathcal Z_H
\end{equation}
inside $K\times B_R\times C_1$.  The remaining amplitude is therefore
supported outside this tube and has fixed derivatives $O(\ee^{CR})$.

The bi-Lipschitz comparison in \cref{lem:nash-gradient} converts
\eqref{eq:critical-tube} into an $\ee^{-CR}$ separation from $\mathcal Z$
in the algebraic $A$-coordinates.  Apply \cref{lem:nash-separation} to the
Nash gradient map $\mathcal F$.  Equations \eqref{eq:height-bound} and
\eqref{eq:critical-tube} give
\begin{equation}\label{eq:gradient-lower}
 \norm{\nabla_k\Theta}\geq \ee^{-CR}.
\end{equation}
Here and below the constant in the exponent may increase.

Let
\[
 V=\frac{\nabla_k\Theta}{\norm{\nabla_k\Theta}^2},
 \qquad V\Theta=1,
\]
and let $V^*$ be its formal adjoint.  Then
\[
 \mathcal Lf=(\ii t)^{-1}V^*f,
 \qquad
 \int f\ee^{\ii t\Theta}=\int \mathcal Lf\,\ee^{\ii t\Theta}.
\]
Every fixed derivative of $V$ is a sum of products of derivatives of
$\Theta$ and powers of $\norm{\nabla_k\Theta}^{-1}$.  Equations
\eqref{eq:phase-bound} and \eqref{eq:gradient-lower} therefore show that a
fixed number of integrations by parts costs at most $C\ee^{CR}$.  Put
$d_0=\sum_{\alpha>0}m_\alpha$ and take $L_1\geq d_0/2+1$ applications.  The
result is $C\ee^{CR}t^{-L_1}$.  Since $\abs{\alpha(H)}\leq CR$,
\[
 \prod_{\alpha>0}(1+t\abs{\alpha(H)})^{-m_\alpha/2}
 \geq C(1+tR)^{-d_0/2}.
\]
For $t,R\geq1$,
$t^{-L_1}\leq C(1+R)^{d_0/2}(1+tR)^{-d_0/2}$, and the polynomial in $R$ is
absorbed by $\ee^{CR}$.  This proves \eqref{eq:local-product} on the
nonstationary part.

\subsection{Conclusion}

The resolved parameter space has fixed dimension.  Its angular variables
and $c\in C_1$ range over fixed compact sets, while its radial coordinate
has length $O(1+R)$ by \eqref{eq:z-domain}.  Since the normal-form radius is
at least $\delta_R=\ee^{-CR}$, a maximal $\delta_R/2$-separated set has
$O((1+R)^d\delta_R^{-d})=O(\ee^{CR})$ points.  The associated cover has
bounded overlap.

Rescaling a fixed bump function to radius $\delta_R$ and dividing by the
sum of the bumps gives a subordinate partition of unity with fixed
derivatives $O_N(\ee^{C_NR})$.  These derivatives are included in the
amplitude bound in \eqref{eq:local-product}.

Each chart therefore satisfies \eqref{eq:local-product}, and there are at
most $C\ee^{CR}$ charts.  Summing them and increasing the constant in the
exponent gives
\[
 \abs{\varphi_{tc}(\exp H)}
 \leq D\ee^{AR}
 \prod_{\alpha\in\Sigma^+}
       (1+t\abs{\alpha(H)})^{-m_\alpha/2}.
\]
This is \eqref{eq:marshall-tracked} and proves \cref{prop:marshall-tracked}.

\begin{remark}[Regularity]
The compact regular set $C_1$ is used in three places in the normal forms:
the description of the critical set, the lower bound for
$\ip c\beta$ in \eqref{eq:Marshall-pivot}, and the proper-Levi argument in
Lemma~4.19 that gives \eqref{eq:Marshall-linear}.  The second and third uses
are precisely the inverse derivatives needed in the quadratic and linear
coordinate changes.
\end{remark}

\section{The automorphic kernel}\label{s:Kernel}

We now give the construction used throughout the paper.  Fix a compact set
$C\Subset\mathfrak a^*_{\rm reg}$ contained in one open Weyl chamber.  We
write the tempered parameter as $Tc$, with $c\in C$ and $T\to\infty$.
When $T=\norm\lambda$, the vector $c$ is simply the unit direction of
$\lambda$.  The boldface symbol $\mathbf c(\lambda)$ below denotes the
Harish--Chandra $c$-function and should not be confused with the direction
$c$.

\begin{proposition}[Positive spherical peak]\label{lem:peak}
For every fixed $0<\epsilon<1$, all $T\geq T_0(C,\epsilon)$, and every
$c\in C$, there is a $K$-bi-invariant Hermitian function
$\Phi_{T,c,\epsilon}\in C_c^\infty(G)$ such that convolution by
$\Phi_{T,c,\epsilon}$ is positive on every unitary representation and
\begin{align}
 \widehat\Phi_{T,c,\epsilon}(Tc)&=1,\label{eq:method-one}\\
 \supp\Phi_{T,c,\epsilon}&\subset
 \{g:d_X(gK,K)\leq R_\epsilon\},
 \label{eq:method-support}\\
 \Phi_{T,c,\epsilon}(e)&\ll_C \epsilon^rT^q,
 \label{eq:method-id}\\
 \sup_{c\in C}\sup_{g\in E}\abs{\Phi_{T,c,\epsilon}(g)}
 &=o_{E,\epsilon}(T^q)
 \qquad(E\Subset G\setminus K).
 \label{eq:method-away}
\end{align}
Here $R_\epsilon$ is independent of $T$ and $c$.
\end{proposition}

\begin{proof}
We use the spherical transform and Paley--Wiener theorem in the normalization
recorded in \cref{s:3}.  Choose a real, even, $W$-invariant Paley--Wiener
bump $\psi_0$ with $\psi_0(0)=1$.  Concretely, take a real, even,
$W$-invariant $h\in C_c^\infty(\mathfrak a)$ with
$\int_{\mathfrak a}h=1$ and let
$\psi_0$ be its Euclidean Fourier transform.  If $\supp h$ is contained in
the radius-$R_0$ ball, then for every $N$
\begin{equation}\label{eq:psi-schwartz}
 \abs{\psi_0(\lambda+\ii\mu)}
 \leq C_N(1+\norm\lambda+\norm\mu)^{-N}
       \ee^{R_0\norm\mu}.
\end{equation}
The translate $\psi_0((\lambda-Tc)/\epsilon)$ has spectral width $\epsilon$
and exponential type $R_0/\epsilon$.  Since the spherical transform must be
$W$-invariant, define
\begin{equation}\label{eq:psi-definition}
 \psi_{T,c,\epsilon}(\lambda)
 =d_{T,c,\epsilon}^{-1}
 \sum_{w\in W}\psi_0\left(\frac{w\lambda-Tc}{\epsilon}\right),
 \qquad
 d_{T,c,\epsilon}
 =\sum_{w\in W}\psi_0\left(\frac{T(wc-c)}{\epsilon}\right).
\end{equation}
Because $c$ ranges over a compact set of regular vectors, there is
$d_C>0$ such that $\norm{wc-c}\geq d_C$ for $w\ne e$.  The Schwartz
bound on the real axis therefore gives, for every $N$,
\begin{equation}\label{eq:d-normalization}
 d_{T,c,\epsilon}=1+O_{N,C}((T/\epsilon)^{-N}).
\end{equation}
In particular, $d_{T,c,\epsilon}$ is bounded away from zero for
$T\geq T_0(C,\epsilon)$.  Reindexing the finite sum in
\eqref{eq:psi-definition} proves $W$-invariance, and substituting
$\lambda=Tc$ shows directly that
\begin{equation}\label{eq:psi-normalized}
 \psi_{T,c,\epsilon}(Tc)=1.
\end{equation}

By spherical Paley--Wiener theory there is a unique
$\beta_{T,c,\epsilon}\in C_c^\infty(K\backslash G/K)$ with
\begin{equation}\label{eq:beta-transform}
 \widehat\beta_{T,c,\epsilon}=\psi_{T,c,\epsilon},
 \qquad
 \supp\beta_{T,c,\epsilon}
 \subset\{g:d_X(gK,K)\leq R_0/\epsilon\}.
\end{equation}
Set
\begin{equation}\label{eq:kernel-construction}
 \Phi_{T,c,\epsilon}
 =\beta_{T,c,\epsilon}^**\beta_{T,c,\epsilon},
 \qquad \beta^*(g)=\overline{\beta(g^{-1})}.
\end{equation}
If $B$ is right convolution by $\beta_{T,c,\epsilon}$, then right
convolution by $\Phi_{T,c,\epsilon}$ is $B^*B$ and is therefore positive.
Moreover,
\begin{equation}\label{eq:Phi-transform}
 \widehat\Phi_{T,c,\epsilon}(\lambda)
 =\abs{\psi_{T,c,\epsilon}(\lambda)}^2\geq0
 \qquad(\lambda\in\mathfrak a^*).
\end{equation}

We derive the four displayed properties in order.

\paragraph{Derivation of \eqref{eq:method-one}.}
Equations \eqref{eq:psi-normalized} and \eqref{eq:Phi-transform} give
$\widehat\Phi_{T,c,\epsilon}(Tc)=\abs{\psi_{T,c,\epsilon}(Tc)}^2=1$.

\paragraph{Derivation of \eqref{eq:method-support}.}
The product support inclusion for convolution and the triangle inequality on
$X$ give
\[
 \supp(\beta^**\beta)
 \subset\{g:d_X(gK,K)\leq2R_0/\epsilon\}.
\]
Thus \eqref{eq:method-support} holds with $R_\epsilon=2R_0/\epsilon$.

\paragraph{Derivation of \eqref{eq:method-id}.}
From \eqref{eq:psi-definition}, Cauchy--Schwarz in the finite Weyl sum, and
\eqref{eq:d-normalization},
\begin{equation}\label{eq:psi-square-bound}
 \abs{\psi_{T,c,\epsilon}(\lambda)}^2
 \leq C_W\sum_{w\in W}
 \abs{\psi_0((w\lambda-Tc)/\epsilon)}^2.
\end{equation}
Spherical inversion at the identity, followed by the change of variables
$w\lambda=Tc+\epsilon u$, therefore gives
\begin{align}
 \Phi_{T,c,\epsilon}(e)
 &=c_X\int_{\mathfrak a^*}\abs{\psi_{T,c,\epsilon}(\lambda)}^2
   \abs{\mathbf c(\lambda)}^{-2}\dd\lambda\notag\\
 &\leq C\epsilon^r\sum_{w\in W}
 \int_{\mathfrak a^*}\abs{\psi_0(u)}^2
 \abs{\mathbf c(w^{-1}(Tc+\epsilon u))}^{-2}\dd u.
 \label{eq:peak-id-proof}
\end{align}
The Plancherel density is $W$-invariant.  On
$\norm u\leq T/(2\epsilon)$, the regular growth estimate
\eqref{eq:c-regular-growth} bounds it by
$C_CT^q(1+\norm u)^q$.  On the complementary region we use the global
bound \eqref{eq:c-global-growth}.  Taking the decay exponent in
\eqref{eq:psi-schwartz} larger than $2r+2q+2$ yields
\begin{align*}
 \epsilon^r\int_{\norm u\leq T/(2\epsilon)}
  \abs{\psi_0(u)}^2\abs{\mathbf c(Tc+\epsilon u)}^{-2}\dd u
 &\leq C_C\epsilon^rT^q,\\
 \epsilon^r\int_{\norm u>T/(2\epsilon)}
  \abs{\psi_0(u)}^2\abs{\mathbf c(Tc+\epsilon u)}^{-2}\dd u
 &\leq C_{C,\epsilon,N}T^{-N}.
\end{align*}
This proves \eqref{eq:method-id}.  The factor $\epsilon^r$ is exactly the
Jacobian of the shrinking window in the $r$ joint spectral coordinates,
while $T^q$ is the regular Plancherel density at $Tc$.

\paragraph{Derivation of \eqref{eq:method-away}.}
Insert \eqref{eq:psi-square-bound} into spherical inversion and make the same
change of variables.  Fix $L>0$.  On $\norm u\leq L$, the normalized
density $T^{-q}\abs{\mathbf c(Tc+\epsilon u)}^{-2}$ is uniformly bounded,
and the regular spherical decay proved in \cref{lem:spherical-decay} gives
\[
 \sup_{c\in C,\,g\in E,\,\norm u\leq L}
 \abs{\varphi_{Tc+\epsilon u}(g)}\longrightarrow0.
\]
The $W$-invariance of both the Plancherel density and the spherical functions
handles every Weyl peak identically.  Thus the portion $\norm u\leq L$,
after division by $T^q$, tends to zero uniformly in $c$ and $g$.  On
$L<\norm u\leq T/(2\epsilon)$, use
$\abs{\varphi_\lambda(g)}\leq1$, \eqref{eq:c-regular-growth}, and Schwartz
decay; the normalized contribution is at most
\[
 C_C\epsilon^r\int_{\norm u>L}
 \abs{\psi_0(u)}^2(1+\norm u)^q\dd u,
\]
which tends to zero as $L\to\infty$, uniformly in $T$.  The remaining region
$\norm u>T/(2\epsilon)$ is $O_{C,\epsilon,N}(T^{-N})$ by
\eqref{eq:c-global-growth}.  Letting first $T\to\infty$ and then
$L\to\infty$ proves \eqref{eq:method-away}.
\end{proof}

\subsection{The logarithmic-scale peak}

\begin{lemma}[Logarithmic spherical peak]\label{lem:log-peak}
Fix $C\Subset\mathfrak a^*_{\mathrm{reg}}$ and $\iota>0$.  There are
$\delta>0$ and $A,B>0$ such that, for
$1\leq R\leq A^{-1}\log T$, the function
$\Phi_{T,c,R^{-1}}$ constructed in \cref{lem:peak} satisfies
\begin{align}
 \Phi_{T,c,R^{-1}}(e)&\leq B T^qR^{-r},
 \label{eq:log-peak-id}\\
 \supp\Phi_{T,c,R^{-1}}&\subset
 \{g:d(gK,K)\leq BR\},\label{eq:log-peak-support}\\
 \abs{\Phi_{T,c,R^{-1}}(g)}&\leq
 B T^{q-\delta}R^{-r}\ee^{A d(gK,K)}
 \label{eq:log-peak-off}
\end{align}
whenever $c\in C$ and
$\iota\leq d(gK,K)\leq BR$.
\end{lemma}

\begin{proof}
Use \eqref{eq:psi-definition} with $\epsilon=R^{-1}$.  The nonidentity terms
in its normalizing denominator are
$\psi_0(TR(wc-c))$, so \eqref{eq:psi-schwartz} gives
\begin{equation}\label{eq:log-normalization}
 d_{T,c,R^{-1}}=1+O_{N,C}((TR)^{-N})
\end{equation}
uniformly for $R\geq1$.  The exponential type of
$\psi_{T,c,R^{-1}}$ is $R_0R$; after
forming $\beta^**\beta$, Paley--Wiener therefore gives
\eqref{eq:log-peak-support} with $B=2R_0$.

For the identity term use \eqref{eq:psi-square-bound} and put
$w\lambda=Tc+u/R$.  The Jacobian is $R^{-r}$.  Split at
$\norm u=TR/2$.  On the inner region, \eqref{eq:c-regular-growth} bounds the
density by $C_CT^q(1+\norm u)^q$; on the outer region,
\eqref{eq:c-global-growth} and Schwartz decay apply.  Consequently
\begin{align*}
 \Phi_{T,c,R^{-1}}(e)
 &\leq C R^{-r}\int_{\mathfrak a^*}\abs{\psi_0(u)}^2
      \abs{\mathbf c(Tc+u/R)}^{-2}\dd u\\
 &\leq C_C T^qR^{-r}+O_N(T^{-N}),
\end{align*}
which is \eqref{eq:log-peak-id} after increasing $B$.

Let $g=k_1\exp Hk_2$ with $\iota\leq\norm H\leq BR$.  In spherical
inversion make the same change of variables and split at
$\norm u=T^{1/4}$.  On the inner region the spectral perturbation is
$u/R$, whose norm is at most $T^{1/4}$, and
\cref{lem:quantitative-decay} gives
\[
 \abs{\varphi_{Tc+u/R}(g)}
 \leq C\ee^{A\norm H}T^{-\delta}.
\]
Together with \eqref{eq:c-regular-growth} this part is at most
\begin{equation}\label{eq:log-peak-inner}
 C T^{q-\delta}R^{-r}\ee^{A\norm H}
 \int_{\mathfrak a^*}\abs{\psi_0(u)}^2(1+\norm u)^q\dd u.
\end{equation}
For $\norm u>T^{1/4}$ we use $\abs{\varphi_\lambda(g)}\leq1$ and
\eqref{eq:c-global-growth}.  Given $N$, the Schwartz bound yields
\begin{equation}\label{eq:log-peak-tail}
 R^{-r}\int_{\norm u>T^{1/4}}\abs{\psi_0(u)}^2
 (1+T+\norm u/R)^q\dd u=O_N(T^{-N}),
\end{equation}
uniformly for $1\leq R\leq A^{-1}\log T$.  The finitely many Weyl peaks
obey identical estimates.  Equations \eqref{eq:log-peak-inner} and
\eqref{eq:log-peak-tail} imply \eqref{eq:log-peak-off}, after reducing
$\delta$ if necessary to absorb the negligible tail and increasing $A,B$.
\end{proof}

\subsection{The quotient kernel and positivity}

Let $\Phi\in C_c^\infty(K\backslash G/K)$.  Its convolution operator on
$L^2(\Gamma\backslash G)^K=L^2(Y)$ has smooth kernel
\begin{equation}\label{eq:quotient-kernel}
 K_\Phi(x,y)=\sum_{\gamma\in\Gamma}
 \Phi(g_x^{-1}\gamma g_y),
\end{equation}
where $x=\Gamma g_xK$ and $y=\Gamma g_yK$.  The sum is locally finite.  If
$\Phi=\beta^**\beta$, the operator is positive.  We now prove directly, without
assuming that $Y$ is compact, the inequality that isolates one eigenfunction.
Let $B$ be right convolution by $\beta$.  Its quotient kernel is
\[
 K_\beta(x,y)=\sum_{\gamma\in\Gamma}
 \beta(g_x^{-1}\gamma g_y).
\]
For fixed $x$, the function $K_\beta(\,cdot\,,x)$ is smooth and belongs to
$L^2(Y)$: it is supported in the image of a fixed metric ball about $x$, and
proper discontinuity makes the defining sum locally finite.  Unfolding the
composition $B^*B$ gives
\begin{equation}\label{eq:kernel-column-norm}
 K_\Phi(x,x)=\int_Y\abs{K_\beta(z,x)}^2\dd z
 =\norm{K_\beta(\,\cdot\,,x)}_2^2.
\end{equation}
There is no convergence issue in the unfolding: before passing to the
quotient, both convolution factors have compact support, so for fixed
$x,z$ only finitely many deck transformations occur.

If $\phi$ is an $L^2$-normalized spherical joint eigenfunction with
parameter $\lambda$, spherical convolution acts on its one-dimensional
joint character by
$B^*\phi=\overline{\widehat\beta(\lambda)}\phi$.  Therefore
\[
 \overline{\widehat\beta(\lambda)}\phi(x)
 =\ip{\phi}{K_\beta(\,\cdot\,,x)}_{L^2(Y)}.
\]
Cauchy--Schwarz, \eqref{eq:kernel-column-norm}, and
$\widehat\Phi=\abs{\widehat\beta}^2$ give
\begin{equation}\label{eq:positivity-isolation}
 \widehat\Phi(\lambda)\abs{\phi(x)}^2\leq K_\Phi(x,x).
\end{equation}
This Hilbert-space proof is the positive pretrace inequality.  It remains
valid in the presence of continuous spectrum because it never expands the
continuous part; it only uses the given $L^2$ eigenfunction.

\subsection{Automorphization and the two scales}

For $x=\Gamma g_xK$, automorphizing the kernel gives the exact, locally
finite diagonal sum
\begin{equation}\label{eq:method-pretrace}
 K_{T,c,\epsilon}(x,x)
 =\sum_{\gamma\in\Gamma}
 \Phi_{T,c,\epsilon}(g_x^{-1}\gamma g_x).
\end{equation}
The Hilbert-space argument in \cref{eq:positivity-isolation} shows that an
$L^2$-normalized joint eigenfunction with parameter $Tc$ satisfies
\begin{equation}\label{eq:method-isolation}
 \abs{\phi(x)}^2\leq K_{T,c,\epsilon}(x,x).
\end{equation}
On a fixed compact subset of $Y$, only uniformly many terms of
\eqref{eq:method-pretrace} occur.  The identity term is
$O(\epsilon^rT^q)$ by \eqref{eq:method-id}, and every nonidentity term is
$o_\epsilon(T^q)$ by \eqref{eq:method-away}.  Hence
\[
 \abs{\phi(x)}^2\leq C\epsilon^rT^q+o_\epsilon(T^q).
\]
Letting first $T\to\infty$ and then $\epsilon\downarrow0$ is the qualitative
mechanism behind \cref{thm:main}.

For \cref{thm:log}, put $R=\kappa\log T$ and $\epsilon=R^{-1}$.  The same
change of variables as in \eqref{eq:peak-id-proof} gives the identity
contribution $O(T^qR^{-r})$.  The quantitative spherical estimate of
\cref{prop:marshall-tracked}, inserted into the preceding construction,
gives \cref{lem:log-peak} and hence
\[
 \sum_{\gamma\ne e}
 \abs{\Phi_{T,c,R^{-1}}(g_x^{-1}\gamma g_x)}
 \ll T^{q-\delta}\ee^{CR}R^{-r}.
\]
When the injectivity radius is bounded below, orbit packing is uniform in
$x$.  Choosing $\kappa$ sufficiently small makes the last display lower
order than $T^qR^{-r}$.  Thus the logarithmic saving comes from shrinking
all $r$ joint spectral coordinates, while regular spherical oscillation pays
for the resulting logarithmic propagation radius.

\section{Proof of the theorems}\label{s:4}

\subsection{Compact sets and uniformly thick quotients}

We first prove the compact-local assertion in \cref{thm:main}.

\begin{proposition}[Thick-part estimate]\label{prop:thick}
Under the spectral hypotheses of \cref{thm:main}, for every $K_0\Subset Y$,
\[
 \norm{\phi_j}_{L^\infty(K_0)}=o_{K_0}(T_j^{q/2}).
\]
\end{proposition}

\begin{proof}
Let $C\Subset\mathfrak a^*_{\mathrm{reg}}$ contain all directions
$c_j=\lambda_j/\norm{\lambda_j}$.  Replacing $T_j$ by $\norm{\lambda_j}$
changes nothing.  Fix $\epsilon>0$ and apply \cref{lem:peak} with
$T=T_j$, $c=c_j$.  By \eqref{eq:positivity-isolation},
\begin{equation}\label{eq:thick-pretrace}
 \abs{\phi_j(x)}^2
 \leq \Phi_{T_j,c_j,\epsilon}(e)
 +\sum_{\substack{\gamma\ne e\\
 d_X(g_xK,\gamma g_xK)\leq R_\epsilon}}
 \abs{\Phi_{T_j,c_j,\epsilon}(g_x^{-1}\gamma g_x)}.
\end{equation}
For $x\in K_0$, \cref{lem:orbit} puts every argument in the second line in
a compact subset of $G\setminus K$ and bounds the number of terms uniformly.
The identity term is at most $C\epsilon^rT_j^q$ by
\eqref{eq:method-id}; the remaining sum is $o_{K_0,\epsilon}(T_j^q)$ by
\eqref{eq:method-away}.  Therefore
\[
 \limsup_{j\to\infty}T_j^{-q}
 \norm{\phi_j}_{L^\infty(K_0)}^2\leq C\epsilon^r.
\]
Let $\epsilon\downarrow0$.  This proves the local assertion.
\end{proof}

Thus the local assertion of \cref{thm:main} holds without any volume
hypothesis.

\subsection{The finite-volume expanding-cusp range}

We use the coordinates and the precisely invariant cusp neighborhood from
the preceding subsection.  We first make the rank-one kernel estimate
explicit.

\begin{lemma}[Rank-one peak kernel]\label{lem:rank-one-kernel}
Fix $0<\epsilon<1$.  For the positive peak of \cref{lem:peak} on
$\mathbb H^{q+1}$,
\begin{equation}\label{eq:rank-one-kernel}
 \abs{\Phi_{T,c,\epsilon}(g)}
 \leq C_{\epsilon}T^q(1+T d_X(gK,K))^{-q/2}
\end{equation}
whenever $g$ belongs to its support.
\end{lemma}

\begin{proof}
In real rank one the sole indivisible positive root has total multiplicity
$q$.  Marshall's estimate \eqref{eq:marshall-tracked}, applied on the fixed
ball of radius $R_\epsilon$, becomes
\begin{equation}\label{eq:rank-one-spherical}
 \abs{\varphi_s(\exp H)}
 \leq C_\epsilon(1+s\norm H)^{-q/2}
 \quad(s\asymp T,\ \norm H\leq R_\epsilon).
\end{equation}
For the peak at $T$, put $s=T+\epsilon u$ in spherical inversion.  On
$\abs u\leq T/(2\epsilon)$ we have $s\asymp T$, so
\eqref{eq:rank-one-spherical}, \eqref{eq:c-regular-growth}, and
Schwartz decay give
\begin{align*}
 &\epsilon\int_{\abs u\leq T/(2\epsilon)}
 \abs{\psi_0(u)}^2\abs{\varphi_{T+\epsilon u}(\exp H)}
 \abs{\mathbf c(T+\epsilon u)}^{-2}\dd u\\
 &\qquad\leq C_\epsilon T^q(1+T\norm H)^{-q/2}
 \int_\mathbb R\abs{\psi_0(u)}^2(1+\abs u)^q\dd u.
\end{align*}
On $\abs u>T/(2\epsilon)$ use $\abs{\varphi_s}\leq1$, the polynomial
density bound \eqref{eq:c-global-growth}, and arbitrarily rapid decay of
$\psi_0$.  That tail is $O_{\epsilon,N}(T^{-N})$, which is absorbed by the
right side of \eqref{eq:rank-one-kernel} because
$\norm H\leq R_\epsilon$.  The Weyl peak at $-T$ has the same estimate.
\end{proof}

We sum \eqref{eq:rank-one-kernel} over nonzero parabolic translations.  Only
vectors with $\abs v\leq C_\epsilon y$ occur, by
\eqref{eq:parabolic-exact-distance} and the support condition.  Write
$\Phi_{T,c,\epsilon}(r)$ for the common value of this radial kernel on
the Cartan double coset at distance $r$.  The number
of lattice vectors in a Euclidean ball of radius $L$ is $O_\Lambda(1+L^q)$.
Vectors with $0<\abs v\leq y/T$ therefore contribute
$O_\epsilon(y^q)$: when $y/T$ is below the shortest lattice length the sum
is empty, and otherwise their number is $O((y/T)^q)$ and each summand is
$O(T^q)$.  For $y/T<\abs v\leq C_\epsilon y$, use
\eqref{eq:cusp-distance-prelim} and decompose into unit Euclidean shells:
\begin{align}
 &\sum_{y/T<\abs v\leq C_\epsilon y}
 T^q(1+T\abs v/y)^{-q/2}\notag\\
 &\qquad\leq C_\epsilon T^{q/2}y^{q/2}
 \sum_{0<\abs v\leq C_\epsilon y}\abs v^{-q/2}
 \leq C_{\epsilon,\Lambda}T^{q/2}y^q.
 \label{eq:parabolic-shell-detail}
\end{align}
The last inequality follows from
$\sum_{k\leq L}k^{q-1-q/2}=O(1+L^{q/2})$.  We have proved
\begin{equation}\label{eq:parabolic-sum}
 \sum_{0\ne v\in\Lambda}
 \abs{\Phi_{T,c,\epsilon}(d((z,y),(z+v,y)))}
 \leq C_\epsilon(y^q+T^{q/2}y^q)
 \leq C_\epsilon T^{q/2}y^q.
\end{equation}

By \eqref{eq:nonparabolic-distance-prelim}, no element outside the cusp
stabilizer can occur in the kernel sum when
$2\log(y/y_0)>R_\epsilon$.  The complementary bounded-height portion of
the cusp, together with the compact core, is a fixed compact subset of $Y$
and is already covered by \cref{prop:thick}.

Fix $\beta<1/2$.  Combining \eqref{eq:method-id},
\eqref{eq:parabolic-sum}, and positivity shows that, uniformly for
$y\leq T^\beta$,
\begin{equation}\label{eq:lower-cusp}
 \abs{\phi(z,y)}^2
 \leq C\epsilon T^q+C_\epsilon T^{q(1/2+\beta)}+o_\epsilon(T^q)
 =C\epsilon T^q+o_\epsilon(T^q).
\end{equation}
Here $q(1/2+\beta)<q$, which is the strict power saving needed for the
parabolic sum.  Divide by $T^q$, first let $T\to\infty$ with $\epsilon$
fixed, and then let $\epsilon\downarrow0$.  Together with the compact-core
estimate in \cref{prop:thick}, this proves
\eqref{eq:truncated-cusp-main} in every one of the finitely many cusps.

\subsection{Proof of the logarithmic improvement}

Assume \eqref{eq:thickness}, and let $\iota>0$ be a uniform lower bound for
the displacement of every nonidentity deck transformation.  Let $\delta,A,B$
be given by \cref{lem:log-peak}, and let $h_X$ be as in
\eqref{eq:orbit-exponential}.  Choose $\kappa>0$ so small that
$\kappa<A^{-1}$ and
\begin{equation}\label{eq:kappa-choice}
 \kappa B(A+h_X)<\frac{\delta}{2},
\end{equation}
and put $R=\kappa\log T$.  Positivity, \cref{lem:log-peak}, and orbit packing
give, uniformly for $x\in Y$,
\begin{align}
 \abs{\phi(x)}^2
 &\leq \Phi_{T,c,R^{-1}}(e)
 +\sum_{\substack{\gamma\ne e\\d(x,\gamma x)\leq BR}}
   \abs{\Phi_{T,c,R^{-1}}(g_x^{-1}\gamma g_x)}\notag\\
 &\ll \frac{T^q}{R^r}
 +\ee^{h_XBR}\frac{T^{q-\delta}}{R^r}\ee^{ABR}
 \ll \frac{T^q}{R^r}.
 \label{eq:log-pretrace-proof}
\end{align}
Indeed, the second term divided by $T^qR^{-r}$ is
$T^{-\delta+\kappa B(A+h_X)}\leq T^{-\delta/2}$ by
\eqref{eq:kappa-choice}.  Taking square roots and using
$R\asymp\log T$ proves \eqref{eq:log-main}.

Every compact quotient satisfies \eqref{eq:thickness}.  A torsion-free
convex cocompact real hyperbolic group is finitely generated and has compact
convex core.  Its translation-length spectrum is discrete away from zero,
so $\ell_0=\inf_{\gamma\ne e}\ell(\gamma)>0$.  Since
$d_X(x,\gamma x)\geq\ell(\gamma)$ for every $x$, one has
$\inj_Y(x)\geq\ell_0/2$ globally.  This proves all assertions of
\cref{thm:log}.

\subsection{Why the proof stops before the turning height}

For completeness, we verify the Airy obstruction advertised in the
introduction.  It also explains why no hidden global finite-volume claim is
contained in \cref{thm:main}.  Define
\[
 I_T(a)=\int_a^\infty\abs{K_{\ii T}(u)}^2\frac{\dd u}{u}.
\]
The uniform Liouville--Green and Airy expansions for imaginary-order Bessel
functions imply, for fixed $a_*>0$ and $0<c_0<c_1<1$,
\begin{align}
 \ee^{\pi T}I_T(a)&\asymp
 \frac{1+\log(T/a)}{T}
 &&(a_*\leq a\leq c_0T),\label{eq:true-bessel-mass}\\
 \ee^{\pi T}\abs{K_{\ii T}(T+sT^{1/3})}^2
 &\asymp T^{-2/3}\abs{\operatorname{Ai}(2^{1/3}s)}^2
 &&(\abs s\leq s_0),
 \label{eq:true-bessel-airy}
\end{align}
where $s_0>0$ is fixed and sufficiently small that the displayed Airy factor
has no zero.  These are direct consequences of Theorems~2.1 and~4.1 of
\cite{Dunster2025}; \eqref{eq:true-bessel-mass} also follows by squaring the
oscillatory expansion on $a\leq u\leq c_1T$.  In that range the phase
derivative with respect to $\log u$ has magnitude
$\sqrt{T^2-u^2}\asymp T$.  Nonstationary phase controls the cosine cross
term, while the nonoscillatory term is $\asymp T^{-1}\dd u/u$; integration
gives
\eqref{eq:true-bessel-mass}.  Formula
\eqref{eq:true-bessel-airy} is the turning-point scaling; at $s=0$ the Airy
factor is nonzero.

Now take a lattice vector $m$ and a height
$y=T/(2\pi\abs m)$, so that its Bessel argument is exactly $T$.  If
$2\pi\abs m y_0\leq c_0T$, equations
\eqref{eq:true-bessel-mass}--\eqref{eq:true-bessel-airy} give
\begin{equation}\label{eq:single-mode-ratio}
 \frac{\abs{K_{\ii T}(2\pi\abs m y)}^2}
 {I_T(2\pi\abs m y_0)}
 \asymp\frac{T^{1/3}}
 {1+\log(T/(2\pi\abs m y_0))}.
\end{equation}
The right side is larger by $T^{1/3}$, up to a logarithm, than the
$O(1)$ ratio predicted by replacing the transition value with its
oscillatory average.  If the shell
$\abs{2\pi\abs m y-T}\leq T^{1/3}$ contains only finitely many lattice
points, spatial lattice summation supplies no averaging that can cancel
\eqref{eq:single-mode-ratio}.  Weighted Cauchy--Schwarz and the Parseval
identity \eqref{eq:cusp-parseval-prelim} therefore do not imply a global
strict Sarnak bound.  Additional arithmetic or geometric information on
the individual Fourier coefficients would be required.  This is the precise
boundary of the present method.

\paragraph{AI use declaration.}
The author used OpenAI's ChatGPT for research brainstorming, literature
discovery, checking intermediate arguments and calculations, and editorial
revision.  ChatGPT was not treated as an author or as an independent source
of mathematical authority.  The author assumes full responsibility for all
statements, proofs, citations, and errors.

\paragraph{Acknowledgements}
We thank Simon Marshall and Peter Sarnak for comments on an earlier draft. In particular, we thank Simon for pointing out the delicacy in making his bound effective.

The author is supported by a grant from the Simons Foundation [SFI-MPS-TSM-00013410]

\begingroup
\fontsize{8.5}{9.5}\selectfont

\endgroup

\end{document}